\documentclass[12pt]{amsart}
\usepackage{amsfonts, amssymb, amsmath, amsthm}
\usepackage[mathscr]{eucal}
\usepackage[hmargin=1.2in, vmargin=1.5in]{geometry}
\input amssym.def
\input amssym

\newtheorem{theorem}{Theorem}[section]
\newtheorem{proposition}[theorem]{Proposition}

\newtheorem{corollary}[theorem]{Corollary}

\theoremstyle{definition}
\newtheorem{definition} [theorem]{Definition}

\theoremstyle{remark} \newtheorem{remark} {Remark}

\numberwithin{equation}{section}

\begin{document}

\title[]{The directional short-time fractional Fourier transform of distributions}

\author[A. Ferizi]{Astrit Ferizi}
\address{Faculty of Mathematics and Natural Sciences, University of Prishtina, George Bush 31, Prishtina, 10000, Kosovo}
\email{ferizi.astrit@gmail.com}

\author[K. Saneva]{Katerina Hadzi-Velkova Saneva}
\address{Faculty of Electrical Engineering and Information Technologies, Ss. Cyril and Methodius University, Rugjer Boshkovikj 18, Skopje, 1000, North Macedonia}
\email {saneva@feit.ukim.edu.mk}

\author[S. Maksimovi\'{c}]{Snje\v{z}ana Maksimovi\'{c}}
\address{Faculty of Architecture, Civil Engineering and Geodesy, University of Banja Luka, Bulevar vojvode Petra Bojovica 1A, Banja Luka, 78000, Bosnia and Herzegovina}
\email{snjezana.maksimovic@aggf.unibl.org}


\begin{abstract}
We introduce  the directional short-time fractional Fourier transform (DSTFRFT)  and prove an extended Parseval's identity and a reconstruction formula for it. We also investigate the continuity of both the directional short-time fractional Fourier transform and its synthesis operator on the appropriate space of test functions. Using the obtained continuity results, we develop a distributional framework for the DSTFRFT on the space of tempered distributions $\mathcal{S}'(\mathbb{R}^n)$. We end the article with a desingularization formula.
\end{abstract}

\maketitle

\section{Introduction}

The Fourier transform (FT) has proven to be a very efficient tool for capturing the frequency information of signals, but it does not provide information about the time at which these frequencies occur  \cite{Gro01}. To overcome this limitation, many transforms have been proposed, such as the short-time Fourier transform (STFT), the wavelet transform (WT), the fractional Fourier transform (FRFT), the  short-time fractional Fourier transform (STFRFT), among others.

The concept of fractional Fourier transforms traces back to at least 1929 (see historical notes and references in \cite{But}). The first comprehensive approach  was introduced  by Kober in 1939 \cite{Kober}.  Since the FRFT depends on the parameter $\alpha$, it can be seen as rotating a signal by an angle $\alpha$ in the time-frequency plane. The FRFT was popularized by Namias in 1980,  as a tool for solving certain classes of ordinary and partial differential equations \cite{Nam80}. A slightly different approach was given by Almeida in \cite{Alm94}, who also provided some properties and an application to the study of swept-frequency filters. 

The FRFT has been extended and studied in different spaces of generalized functions using two different techniques, one analytic and the other algebraic \cite{Zay07}. The FRFT theory of tempered distributions can be found in \cite{Path12} where it is shown that the FRFT is an isomorphism of the space of tempered distribution $\mathcal{S}'(\mathbb{R})$ onto itself. The \textit{n}-dimensional FRFT is defined by using the tensor product of \textit{n} copies of the one-dimensional FRFT and has been investigated by different authors \cite{Kam19, Oza01, Tof23}. A different approach is given by Zayed in \cite{Zay18}, who defines a two-dimensional FRFT that is not a tensor product of the one-dimensional transform.

The STFT has been shown to be an efficient tool for capturing both the time and frequency information of a signal \cite{Gro01}. A natural generalization of the STFT, called the short-time fractional Fourier transform (STFRFT), has been introduced in \cite{Tao10}. Using the tensor product of \textit{n}-copies of the one-dimensional STFRFT, Gao and Li introduced a multi-dimensional STFRFT \cite{Gao21}.
Following the duality approach, the STFRFT has recently been extended to the space of tempered distributions and has been used to study a significant part of the theory of distributions, the asymptotic behavior of distributions \cite{Ata23}.

In \cite{Can99}, Candes introduced a new transform for dealing with higher-dimensional phenomena, called the ridgelet transform, which is a directional-sensitive variant of the one-dimensional WT. The ridgelet transform first acts on multidimensional signals with the Radon transform (RT) and then applies the one-dimensional WT. Similarly, Grafakos and Sunsing introduced a directional sensitive variant of the STFT, which is a composition of the RT with the one-dimensional STFT \cite{Gra08}. Giv used a different approach in \cite{Giv13}, defining the directional short-time Fourier transform (DSTFT) of $f\in L^1(\mathbb{R}^n)$ with respect to $\psi \in L^{\infty}(\mathbb{R})$ as 
\begin{equation} \label{Giv}
DS_{\psi}f(\textbf{u},b,\textbf{a})=(2\pi)^{-n/2} \int_{\mathbb{R}^n}{f(\textbf{x}) \overline{\psi}(\textbf{x}\cdot \textbf{u}-b)e^{-i\textbf{x}\cdot \textbf{a}}}d\textbf{x},
\end{equation}
where $(\textbf{u},b,\textbf{a})\in \mathbb{S}^{n-1}\times \mathbb{R} \times \mathbb{R}^n$ ($\mathbb{S}^{n-1}$ is the unit sphere in $\mathbb{R}^n$). Following the duality approach, the authors of  \cite{San16}  extended and studied the DSTFT on the space of tempered distributions $\mathcal{S}'(\mathbb{R}^n)$. They also proved that the largest distribution space in which the direct approach works is $D_{L^1}'(\mathbb{R}^n)$.

Motivated by the  DSTFT defined by Giv, in this paper, we introduce a transform that extends the capabilities of the STFRFT for more efficiently probing  the directional properties of signals. We called this directional sensitive variant of the STFRFT the directional short-time fractional Fourier transform (DSTFRFT). This new transform may be considered as a generalization of the DSTFT defined by (\ref{Giv}) (see Remark \ref{rem1} below).

The paper is organized as follows. In Section \ref{Se2}, the DSTFRFT and its synthesis operator are introduced, and the extended Parseval's identity and the reconstruction formula are proved. 
In Section \ref{se3}, the continuity of the DSTFRFT (Theorem \ref{ctdstfrft}) and its synthesis operator (Theorem \ref{ctdstfrft1}) on the spaces $\mathcal{S}(\mathbb{R}^n)$ and $\mathcal{S}(\mathbb{Y}^{2n})$, $\mathbb{Y}^{2n}=\mathbb{S}^{n-1} \times \mathbb{R} \times \mathbb{R}^{n}$, respectively, has been proven. The continuity results in Section \ref{se4} have been used to develop a framework for the DSTFRFT on $\mathcal{S}'(\mathbb{R}^n)$.  It is shown  that both the  DSTFRFT and its synthesis operator can be extended as continuous
mappings $DS_{\psi}^{\alpha}:\mathcal{S}' (\mathbb{R}^n) \to \mathcal{S}'(\mathbb{Y}^{2n}) $ and  $(DS_\psi ^{\alpha}) ^{*}:\mathcal{S}'(\mathbb{Y}^{2n}) \to \mathcal{S}' (\mathbb{R}^n) $  for every $\psi \in \mathcal{S}(\mathbb{R})$.
In the last section, Section \ref{se5}, the relationship between the FRFT and the DSTFRFT has been shown, and a desingularization formula has been obtained.


\section{Preliminaries}

\subsection{Notations and spaces}
 We briefly introduce the \textit{n}-dimensional notation. When $\textbf{x}=(x_1,x_2,...,x_n) \in \mathbb{R}^{n}, \textbf{y}=(y_1,y_2,...,y_n) \in \mathbb{R}^n$ and $ \alpha=(\alpha_1,\alpha_2,...,\alpha_n) \in \mathbb{N}_{0}^{n}$, we write $\textbf{x}^\alpha=x_{1}^{\alpha_1}x_{2}^{\alpha_2}\cdots x_{n}^{\alpha_n}$, 
$\partial_{\textbf{x}}^{\alpha}=\partial^{\alpha_1}_{x_1}\partial^{\alpha_2}_{x_2}\cdots \partial^{\alpha_n}_{x_n}=\frac{\partial^{|\alpha|}}{\partial x_{1}^{\alpha_1} \partial x_{2}^{\alpha_2} \cdots \partial x_{n}^{\alpha_n}}$, $|\alpha|=\alpha_1+\alpha_2+...+\alpha_n$, $\vert \textbf{x} \vert$ denotes the Euclidean norm and $\textbf{x}\cdot \textbf{y}= x_1 y_1+x_2 y_2+\cdots +x_n y_n$ is the scalar product of $\textbf{x}$ and $\textbf{y}$. We write $A\lesssim B$ when $A \leq C \cdot B $ for some positive constant $C$.
We also  use the notation $\left( f, \varphi \right)$ for the $L^2$-inner product of $f$ and $\varphi$ and $\langle f, \varphi \rangle$ for the dual pairing between the distribution $f$ and a test function $\varphi$; $\left( f, \varphi \right)=\langle f, \overline{\varphi} \rangle$.
 The FT $\mathcal{F}f$ of a function $f \in L^{1}(\mathbb{R}^n)$ is defined as 
$\mathcal{F}f(\boldsymbol{\xi})=\widehat{f}(\boldsymbol{\xi})=(2\pi)^{-n/2} \int_{\mathbb{R}^n}f(\textbf{x})e^{-i\textbf{x}\cdot \boldsymbol{\xi}}d\textbf{x}, \enspace  \boldsymbol{\xi} \in \mathbb{R}^n,$
and it is extended to $ L^2(\mathbb{R}^n)$ in the usual way \cite{Hor83}.

The well-known Schwartz space  $\mathcal{S}(\mathbb{R}^n)$ consists of all smooth complex-valued functions $\varphi\in C^{\infty}(\mathbb{R}^n)$ such that 
\begin{equation} \label{Schartzspace}
\rho_m(\varphi)=\sup_{\textbf{x}\in \mathbb{R}^n,   |\alpha| \leq m } (1+|\textbf{x}|)^{m} | \partial^{\alpha}\varphi(\textbf{x}) | <\infty,
\end{equation}
for all $m\in \mathbb{N}_{0}$ \cite{Sch66, Hor83}. It is topologized by means of seminorms (\ref{Schartzspace}). Its dual $\mathcal{S}^{'}(\mathbb{R}^n)$ is the well-known space of tempered distributions. All dual spaces in the paper are equipped with the strong dual topology \cite{Tre67}.

{The STFT $S_{\psi}f$ of an integrable function $f\in L^1(\mathbb{R}^n)$ with respect to a window function $\psi\in \mathcal{S}(\mathbb{R}^n)$ is defined as 
\begin{equation} \label{STFT}
    S_{\psi}f(\textbf{b},\textbf{a})=(2\pi)^{-n/2}  \int_{\mathbb{R}^n}{f(\textbf{x})\overline{\psi}(\textbf{x}-\textbf{b}) e^{-i\textbf{x}\cdot \textbf{a}}d\textbf{x}},
\end{equation}
    where $(\textbf{b},\textbf{a})\in \mathbb{R}^n \times \mathbb{R}^n$ \cite{Gro01}. } 

Let $\mathbb{Y}^{2n}=\mathbb{S}^{n-1} \times \mathbb{R} \times \mathbb{R}^{n}$, $n \geq 2$. The unit sphere $\mathbb{S}^{n-1}\subset \mathbb{R}^{n}$ will be equipped with normalized surface area measure. 
We introduce  the space $\mathcal{S}(\mathbb{Y}^{2n})$ of all smooth functions $\Phi \in C^{\infty}(\mathbb{Y}^{2n})$ such that
\begin{equation} \label{S(Y^(2n))}
\rho_{s,r}^{l,m,k}(\Phi)=\sup_{(\textbf{u},b,\textbf{a})\in \mathbb{Y}^{2n} }{(1+|b|^2)^{r/2} (1+|\textbf{a}|^2)^{s/2} |\partial_{\textbf{a}}^{l}  \partial_{b}^{m} \Delta_{\textbf{u}}^{k} \Phi(\textbf{u},b,\textbf{a})|}<\infty,
\end{equation}
for all $s,r,m,k \in \mathbb{N}_0$, $l\in \mathbb{N}_0^{n}$, where $\Delta_{\textbf{u}}$ stands for the Laplace-Beltrami operator on the unit sphere $\mathbb{S}^{n-1}$. The topology of $\mathcal{S}(\mathbb{Y}^{2n})$ is defined by the family of seminorms (\ref{S(Y^(2n))}).
Its dual space is denoted by $\mathcal{S}'(\mathbb{Y}^{2n})$. We fix $dbd\textbf{a}d\textbf{u}$ as a standard measure on $\mathbb{Y}^{2n}$.
If $F$ is a locally integrable function of slow growth on $\mathbb{Y}^{2n}$, i.e. if there exist $C>0$ and $s\in \mathbb{N}_0$ such that 
$$|F(\textbf{u},b,\textbf{a})| \leq C \left( 1+|\textbf{a}| \right)^s \left( 1+|b| \right)^s,  \enspace (\textbf{u},b,\textbf{a})\in \mathbb{Y}^{2n}, $$
then $F$ will be identified with an element of $\mathcal{S}'(\mathbb{Y}^{2n})$ via the action 
\begin{equation} \label{standardidentification}
\langle F,\Phi \rangle:= \int_{\mathbb{S}^{n-1}}\int_{\mathbb{R}^{n}} \int_{\mathbb{R}} F(\textbf{u},b,\textbf{a})\Phi(\textbf{u},b,\textbf{a}) dbd\textbf{a}d\textbf{u},  \enspace \Phi \in \mathcal{S}(\mathbb{Y}^{2n}).
\end{equation}

The nuclearity of the Schwartz spaces (the Schwartz kernel theorem) yields the following equality $\mathcal{S}(\mathbb{Y}^{2n}) = \mathcal{S}(\mathbb{S}^{n-1} \times \mathbb{R}) \hat{\otimes} \mathcal{S}(\mathbb{R}^{n})$, where $X \hat{\otimes} Y$ is the topological tensor product space obtained as the completion of $X \otimes Y$ in the $\pi$-topology or, equivalently in the $\epsilon$-topology \cite{Tre67, Her83}. It also leads to the following isomorphisms $\mathcal{S}'(\mathbb{Y}^{2n}) \cong \mathcal{S}'(\mathbb{S}^{n-1} \times \mathbb{R}, \mathcal{S}'(\mathbb{R}^{n})) \cong \mathcal{S}'(\mathbb{R}^{n}, \mathcal{S}'(\mathbb{S}^{n-1} \times \mathbb{R}))$ which is being realized via the standard identification
\begin{equation} \label{SI}
    \langle F, \varphi \otimes \Psi \rangle=\langle \langle F, \varphi \rangle, \Psi \rangle=\langle \langle F, \Psi \rangle, \varphi \rangle, \quad \varphi \in \mathcal{S}(\mathbb{R}^{n}), \, \Psi \in \mathcal{S}(\mathbb{S}^{n-1} \times \mathbb{R}).
\end{equation}


\subsection{The RT }

The RT $Rf$ of a function $f \in L^{1}(\mathbb{R}^n)$ is defined as 
$$
Rf_{\textbf{u}}(p)=\int_{\textbf{x} \cdot \textbf{u} =p} f(\textbf{x})d\textbf{x}= \int_{\mathbb{R}^n} f(\textbf{x}) \delta(p-\textbf{x}\cdot \textbf{u} )d\textbf{x},
$$
where $\textbf{u}\in \mathbb{S}^{n-1}$, $p\in \mathbb{R}$ and $\delta$ is the Dirac-delta function \cite{Her83}. By the Fubini's theorem, one can show that $Rf \in L^1(\mathbb{S}^{n-1} \times \mathbb{R})$ for $f\in L^1(\mathbb{R}^n)$. The dual RT (or back-projection) $R^*\varrho$ of a function $\varrho \in L^{\infty}(\mathbb{S}^{n-1}\times \mathbb{R})$ is defined as 
$$
R^*\varrho(\textbf{x})=\int_{\mathbb{S}^{n-1}} \varrho(\textbf{u},\textbf{x} \cdot \textbf{u})d\textbf{u},
$$
where $\textbf{x}\in \mathbb{R}^n$ \cite{Her83}.

The FT and the RT are connected by the so-called Fourier slice theorem, which states that for $f \in L^{1}(\mathbb{R}^n)$ is true \begin{equation*} \widehat{Rf_\textbf{u}}(p)=(2\pi)^{\frac{n-1}{2}} \widehat{f}(p\textbf{u}),\end{equation*}
where $\textbf{u}\in \mathbb{S}^{n-1}$ and $p\in \mathbb{R}$ \cite{Her83, Hel99}.


\subsection{The FRFT and the STFRFT}
In this subsection, we give definitions and basic properties of the FRFT and the STFRFT. For more details, we refer to \cite{Oza01, Alm94, Tao10, Zay18, Gao21, Path12, Tof23, Kam19, Zay07}.

The FRFT $\mathcal{F}_{\alpha}f$ of order  $\alpha=(\alpha_1,\alpha_2,...,\alpha_n)\in \mathbb{R}^n$ of a function $f \in L^{1}(\mathbb{R}^n)$ is defined as 
\begin{equation} \label{FrFT}
    \mathcal{F}_{\alpha} f(\boldsymbol{\xi})=\int_{\mathbb{R}^n}f(\textbf{x})K_{\alpha}(\textbf{x},\boldsymbol{\xi})d\textbf{x}, \enspace  \boldsymbol{\xi}\in \mathbb{R}^n,
\end{equation}
where 
\begin{equation}\label{kernel} K_{\alpha}(\textbf{x},\boldsymbol{\xi})=\prod_{k=1}^{n} K_{\alpha_k}(x_k,\xi_k),\end{equation} 
$\textbf{x}=(x_1,...,x_n),\ \boldsymbol{\xi}=(\xi_1,...,\xi_n)$ and 
$K_{\alpha_k}(x_k,\xi_k)$ are defined by
\begin{align*}
K_{\alpha_k}(x_k,\xi_k)&=\begin{cases}
C_{\alpha_k} e^{i \left( \frac{x_k^2+\xi_{k}^2}{2}c_{1}^{(k)}-x_k \xi_k c_{2}^{(k)}  \right)} &if \ \alpha_k \notin \pi \mathbb{Z}  \\
\delta(x_k-\xi_k) &if \ \alpha_k \in 2\pi \mathbb{Z} \\
\delta(x_k+\xi_k) &if \ \alpha_k \in 2 \pi \mathbb{Z} +\pi
\end{cases},
\end{align*} 
$c_{1}^{(k)}=\text{cot}(\alpha_k),\ c_{2}^{(k)}=\text{csc}(\alpha_k),\ C_{\alpha_k}=\sqrt{\frac{1-ic_{1}^{(k)}}{2\pi}}$ \cite{Zay18, Oza01, Alm94}. Throughout the paper, we employ the following notations
$C_{\alpha}=\prod_{k=1}^{n} C_{\alpha_k}, \ c_1=(c_1^{(1)},...,c_1^{(n)})$ and $c_2=(c_2^{(1)},...,c_2^{(n)})$. 

If we put $\alpha=(\alpha_1,\alpha_2,...,\alpha_n)=(\pi/2,\pi/2,...,\pi/2)$ in (\ref{FrFT}), we obtain the definition of the FT. If $f\in L^1(\mathbb{R}^n)$ sucht that $\mathcal{F}_{\alpha}f \in L^1(\mathbb{R}^n)$ then 
\begin{equation} \label{invFRFT}
    f(\textbf{w})=\int_{\mathbb{R}^{n}} \mathcal{F}_{\alpha}  f  \left( \boldsymbol{\xi} \right) K_{-\alpha}(\textbf{w},\boldsymbol{\xi})  d\boldsymbol{\xi}, 
\end{equation}
for a.e. \textbf{w} $\in \mathbb{R}^n$ (\cite{Kam19}, Thm. 3.11).
It is shown in \cite{Path12, Zay07} that the FRFT is a continuous mapping from $\mathcal{S}(\mathbb{R}^n)$ onto itself and following the duality approach it can be extended to the space of tempered distributions by 
$$ \langle \mathcal{F}_{\alpha} f, \varphi \rangle = \langle f, \mathcal{F}_{\alpha} \varphi \rangle, \ f\in \mathcal{S}'(\mathbb{R}^n), \ \varphi \in \mathcal{S}(\mathbb{R}^n).$$

The STFRFT $S_{\psi}^{\alpha}f$ of order $\alpha=(\alpha_1,\alpha_2,...,\alpha_n)\in \mathbb{R}^n$ of an integrable function $f\in L^1(\mathbb{R}^n)$ with respect to a window function $\psi\in \mathcal{S}(\mathbb{R}^n)$ is defined as 
\begin{equation} \label{STFrFT}
    S_{\psi}^{\alpha}f(\textbf{b},\textbf{a})= \int_{\mathbb{R}^n}{f(\textbf{x})\overline{\psi}(\textbf{x}-\textbf{b}) K_{\alpha}(\textbf{x},\textbf{a})d\textbf{x}},
\end{equation}
where $(\textbf{b},\textbf{a})\in \mathbb{R}^n \times \mathbb{R}^n$ and  $K_{\alpha}(\textbf{x},\textbf{a})$ is given by (\ref{kernel}) (\cite{Gao21}, Def. 2).
For $\alpha=(\alpha_1,\alpha_2,...,\alpha_n)=(\pi/2,...,\pi/2)$ in (\ref{STFrFT}), we obtain the definition of the STFT (\ref{STFT}).

Let $\psi\in \mathcal{S}(\mathbb{R}^n)$ be a non-trivial window and let $\eta \in \mathcal{S}(\mathbb{R}^n) $ be a synthesis window for $\psi$, namely, $\left( \eta, \psi \right) \neq 0.$ If $f\in L^1(\mathbb{R}^n)$ sucht that $\mathcal{F}_{\alpha}f \in L^1(\mathbb{R}^n) $, then the following reconstruction formula holds pointwisely,
\begin{equation*}f(\normalfont\textbf{w})=\frac{1}{\left( \eta, \psi \right)} \int_{\mathbb{R}^{n}} \int_{\mathbb{R}^n} S_\psi ^{\alpha} f(\textbf{b},\textbf{a}) \eta(\textbf{w} -\textbf{b}) 	K_{-\alpha}(\textbf{w},\textbf{a}) d\textbf{b}d\textbf{a},\end{equation*}
for a.e. $\textbf{w}\in \mathbb{R}^n$ \cite{Tao10, Ata23}. 


\section{The DSTFRFT of functions}\label{Se2}

In this section we define the DSTFRFT and prove an extended Parseval's identity and a reconstruction formula that suggests the definition of the directional short-time fractional Fourier synthesis operator. 
\begin{definition} \label{def3.1}
Let $\psi\in \mathcal{S}(\mathbb{R})$. The DSTFRFT $DS_{\psi}^{\alpha}f$ of order $\alpha=(\alpha_1,\alpha_2,...,\alpha_n)\in \mathbb{R}^n$ of an integrable function $f\in L^1(\mathbb{R}^n)$ with respect to a window function  $\psi$ is defined as 
\begin{equation} \label{DSTFrFT}
\begin{split}
DS_{\psi}^{\alpha}f(\textbf{u},b,\textbf{a}):&=\left( f(\textbf{x}), \psi(\textbf{u}\cdot \textbf{x}-b)K_{-\alpha}(\textbf{x},\textbf{a}) \right)\\
&=  \int_{\mathbb{R}^n}{f(\textbf{x})\overline{\psi}(\textbf{u}\cdot \textbf{x}-b) K_{\alpha}(\textbf{x},\textbf{a})d\textbf{x}},
\end{split}
\end{equation}
where $(\textbf{u},b,\textbf{a})\in \mathbb{Y}^{2n}$ and  $K_{\alpha}(\textbf{x},\textbf{a})$ is given by (\ref{kernel}). Since $K_{\alpha_k}(x_k,a_k)$ is a $2\pi$-periodic function with respect to $\alpha_k$, from now and on, we suppose that $\alpha_k \in (-\pi,\pi) \setminus \lbrace 0 \rbrace $ for $k\in \lbrace 1,2,...,n \rbrace$.
\end{definition}

\begin{remark} \label{rem1}
If we put $\alpha=(\alpha_1,\alpha_2,...,\alpha_n)=(\pi/2,...,\pi/2)$ in (\ref{DSTFrFT}), we obtain the  DSTFT (\ref{Giv}) defined  by Giv \cite{Giv13}. 
\end{remark}

\begin{remark}
Since it may happen that $\overline{\psi}(\textbf{u}\cdot \textbf{x}-b)K_{\alpha}(\textbf{x},\textbf{a}) \notin D_{L^{\infty}}(\mathbb{R}^n) $ then Definition \ref{def3.1} does not work for $f\in D_{L^1}'(\mathbb{R}^n)$. It also does not work for $f\in \mathcal{S} '(\mathbb{R}^n)$.
\end{remark}

Note that
$$\quad DS_{\psi}^{\alpha} f(\textbf{u},b,\textbf{a})=\mathcal{F}_{\alpha} \left( f(\cdot)\overline{\psi}((\cdot)\cdot \textbf{u}-b) \right)(\textbf{a}), \quad (\textbf{u},b,\textbf{a})\in\mathbb Y^{2n}.$$

The following proposition presents a useful relation between the DSTFRFT and the FT.

\begin{proposition} \label{prop3.2}
For $\psi\in \mathcal{S}(\mathbb{R})$ and $f\in L^1(\mathbb{R}^n)$ is true
\begin{align*} 
DS_{\psi}^{\alpha} f(\textbf{u},b,\textbf{a})
&=(2\pi)^{\frac{n-1}{2}} \int_{\mathbb{R}}{\mathcal{F} \left( f(\cdot)  K_{\alpha}(\cdot,\textbf{a})\right) \left( \xi \textbf{u} \right)  \overline{\mathcal{F}{\psi}\left( \xi \right)}  e^{i \xi b} d \xi}.\\
&=(2\pi)^{-\frac{1}{2}} \int_{\mathbb{R}}{\mathcal{F}_{\alpha} \left( f(\cdot) e^{-i(\cdot)\cdot (\xi \textbf{u})} \right) \left( \textbf{a} \right)  \overline{\mathcal{F}{\psi}\left( \xi \right)}  e^{i \xi b} d \xi}\\
&=(2\pi)^{\frac{n}{2}} \mathcal{F}\left({\mathcal{F} \left( f(\cdot)  K_{\alpha}(\cdot,\textbf{a}) \right)\left((\cdot) \textbf{u} \right) \overline{\mathcal{F}{\psi}\left( \cdot \right)} }\right)(-b)
\end{align*}
where $(\textbf{u}, b,\textbf{a})\in\mathbb{Y}^{2n}.$
\end{proposition}

\begin{proof}
By the Fubini's theorem, the Parseval's identity for the FT and the Fourier slice theorem, we obtain
\begin{align*}
 \quad DS_\psi ^{\alpha} f(\textbf{u},b,\textbf{a})&= \int_{\mathbb{R}^n}{f(\textbf{x})\overline{\psi}(\textbf{x} \cdot \textbf{u}-b) K_{\alpha}(\textbf{x},\textbf{a})}d\textbf{x}\\
&= \int_{\mathbb{R}^n}{f(\textbf{x})  K_{\alpha}(\textbf{x},\textbf{a})d\textbf{x} \int_{\mathbb{R}} \overline{\psi}(p-b) \delta (p-\textbf{x} \cdot \textbf{u})dp }\\
&= \int_{\mathbb{R}}{\overline{\psi}(p-b) dp \int_{\mathbb{R}^n} f(\textbf{x})  K_{\alpha}(\textbf{x},\textbf{a}) \delta (p-\textbf{x} \cdot \textbf{u})d\textbf{x} }\\
&= \int_{\mathbb{R}}{R\left( f(\cdot) K_{\alpha} (\cdot,\textbf a)\right)_{\textbf{u}}(p)\cdot \overline{\psi}(p-b) dp }\\
&=(2\pi)^{\frac{n-1}{2}} \int_{\mathbb{R}}{\mathcal{F} \left( f(\cdot)  K_{\alpha}(\cdot,\textbf{a})\right) \left( \xi \textbf{u} \right)  \overline{\mathcal{F}{\psi}\left( \xi \right)}  e^{i \xi b} d \xi}.
\end{align*}
\end{proof}

The next proposition gives an extended Parseval's identity.

\begin{proposition}\textnormal{(Extended Parseval's relation)} \label{prop3.4}
Let $\psi\in \mathcal{S}(\mathbb{R})$ be a non-trivial window and let $\eta \in \mathcal{S}(\mathbb{R}) $ be a synthesis window for it. Then 
$$\int_{\mathbb{R}^n} f(\normalfont\textbf{x}) \overline{g(\textbf{x})}d\textbf{x}=\frac{1}{\left( \eta, \psi \right)} \int_{\mathbb{S}^{n-1}}\int_{\mathbb{R}^{n}} \int_{\mathbb{R}} DS_\psi ^{\alpha} f(\textbf{u},b,\textbf{a}) \overline{DS_{\eta} ^{\alpha} g(\textbf{u},b,\textbf{a})} dbd\textbf{a}d\textbf{u},$$
for $f,g \in L^1(\mathbb{R}^n)\cap L^2(\mathbb{R}^n)$.
\end{proposition}
\begin{proof} First we can see that for $(\textbf{u},{\textbf{a}})\in \mathbb{S}^{n-1} \times \mathbb{R}^n$,
\begin{align*}
&\mathcal{F} \left( f(\cdot)  K_{\alpha}(\cdot,\textbf{a})\right) \left( \xi \textbf{u} \right) \overline{\mathcal{F}{\psi}\left( \xi \right)}\in L^1(\mathbb{R})\cap L^2(\mathbb{R}), \text{and} \\ &\mathcal{F} \left( g(\cdot)  K_{\alpha}(\cdot,\textbf{a})\right) \left( \xi \textbf{u} \right) \overline{\mathcal{F}{\eta}\left( \xi \right)} \in L^1(\mathbb{R})\cap L^2(\mathbb{R}).
\end{align*}
Using Proposition \ref{prop3.2}, the Parseval's identity for the FRFT (\cite{Kam19}, Thm. 3.3), and the Fubini's theorem, we obtain
\begin{align*} 
&\, \quad \int_{\mathbb{S}^{n-1}}\int_{\mathbb{R}^{n}} \int_{\mathbb{R}} DS_\psi ^{\alpha} f(\textbf{u},b,\textbf{a}) \overline{DS_{\eta} ^{\alpha} g(\textbf{u},b,\textbf{a})} dbd\textbf{a}d\textbf{u} \\
&=(2\pi)^n \int_{\mathbb{S}^{n-1}}\int_{\mathbb{R}^{n}}d\textbf{a}d\textbf{u}  \int_{\mathbb{R}} \mathcal{F} \left( \mathcal{F} \left( f(\cdot) K_{\alpha}(\cdot,\textbf a) \right)\left( (\cdot) \textbf{u} \right) \overline{\mathcal{F}{\psi}\left( \cdot \right)} \right)(-b) \\
&\quad \times \overline{\mathcal{F} \left( \mathcal{F} \left( g(\cdot) K_{\alpha}(\cdot,\textbf a) \right)\left( (\cdot) \textbf{u} \right) \overline{\mathcal{F}{\eta}\left( \cdot \right)} \right)(-b)}  db\\
&=(2\pi)^n  \int_{\mathbb{S}^{n-1}}\int_{\mathbb{R}^{n}}d\textbf{a}d\textbf{u}  \int_{\mathbb{R}}  \mathcal{F} \left( f(\cdot)  K_{\alpha}(\cdot,\textbf{a})\right) \left( \xi \textbf{u} \right)  \overline{\mathcal{F} \left( g(\cdot)  K_{\alpha}(\cdot,\textbf{a})\right) \left( \xi \textbf{u} \right) } \mathcal{F}{\eta}\left( \xi \right) \overline{\mathcal{F}{\psi}\left( \xi \right)} d \xi\\
&=\int_{\mathbb{S}^{n-1}} d\textbf{u} \int_{\mathbb{R}}\overline{\mathcal{F}{\psi}\left( \xi \right)} \mathcal{F}{\eta}\left( \xi \right)  d \xi  \int_{\mathbb{R}^n}  \mathcal{F}_{\alpha} \left( f(\cdot) e^{-i(\cdot)\cdot (\xi \textbf{u})} \right) \left( \textbf{a} \right) \overline{ \mathcal{F}_{\alpha} \left( g(\cdot) e^{-i(\cdot)\cdot (\xi \textbf{u})} \right) \left( \textbf{a} \right) } d\textbf{a}  \\
&= \int_{\mathbb{S}^{n-1}} d\textbf{u} \int_{\mathbb{R}}\overline{\psi \left( \xi \right)} \eta\left( \xi \right)  d \xi  \int_{\mathbb{R}^n}   f(\textbf{w}) \overline{  g(\textbf{w}) }  d\textbf{w} \\
&=\left( \eta, \psi \right)  \left( f,g \right).
\end{align*}
\end{proof}
The next proposition gives the reconstruction formula for the DSTFRFT.

\begin{proposition}\textnormal{(Reconstruction formula)} \label{prop3.5}
Let $\psi\in \mathcal{S}(\mathbb{R})$ be a non-trivial window and let $\eta \in \mathcal{S}(\mathbb{R}) $ be a synthesis window for it. If $f\in L^1(\mathbb{R}^n)$ such that $\mathcal{F}_{\alpha}f \in L^1(\mathbb{R}^n) $, then the following reconstruction formula holds pointwisely,
\begin{equation*}f(\normalfont\textbf{w})=\frac{1}{\left( \eta, \psi \right)} \int_{\mathbb{S}^{n-1}}\int_{\mathbb{R}^{n}} \int_{\mathbb{R}} DS_\psi ^{\alpha} f(\textbf{u},b,\textbf{a}) \eta(\textbf{w} \cdot \textbf{u}-b) 	K_{-\alpha}(\textbf{w},\textbf{a}) dbd\textbf{a}d\textbf{u},   \enspace a.e. \quad  \textbf{w}\in \mathbb{R}^n.\end{equation*}
\end{proposition}

\begin{proof} 
Using Proposition \ref{prop3.2}, the formula (\ref{invFRFT}) and the Fubini's theorem, we obtain
\begin{align*}
&\, \quad \int_{\mathbb{S}^{n-1}}\int_{\mathbb{R}^{n}} \int_{\mathbb{R}} DS_\psi ^{\alpha} f(\textbf{u},b,\textbf{a}) \eta(\textbf{w} \cdot \textbf{u}-b) 	K_{-\alpha}(\textbf{w},\textbf{a}) dbd\textbf{a}d\textbf{u}\\
&=(2\pi)^{-\frac{1}{2}} \int_{\mathbb{S}^{n-1}} \int_{\mathbb{R}^{n}} K_{-\alpha}(\textbf{w},\textbf{a}) d\textbf{a}d\textbf{u} \int_{\mathbb{R}}  \eta(\textbf{w} \cdot \textbf{u}-b) db \int_{\mathbb{R}}{\mathcal{F}_{\alpha} \left( f(\cdot) e^{-i(\cdot)\cdot (\xi \textbf{u})} \right) \left( \textbf{a} \right)  \overline{\mathcal{F}{\psi}\left( \xi \right)}  e^{i \xi b} d \xi} \\
&=(2\pi)^{-\frac{1}{2}} \int_{\mathbb{S}^{n-1}} \int_{\mathbb{R}} \overline{\mathcal{F}{\psi}\left( \xi \right)} d \xi d\textbf{u} \int_{\mathbb{R}^{n}} \mathcal{F}_{\alpha} \left( f(\cdot) e^{-i(\cdot)\cdot (\xi \textbf{u})} \right) \left( \textbf{a} \right) K_{-\alpha}(\textbf{w},\textbf{a})  d\textbf{a} \int_{\mathbb{R}}{\eta(\textbf{w} \cdot \textbf{u}-b)   e^{i \xi b} db} 
\end{align*}
\begin{align*}
&=\int_{\mathbb{S}^{n-1}} \int_{\mathbb{R}} \overline{\mathcal{F}{\psi}\left( \xi \right)} \mathcal{F}{\eta}\left( \xi \right) e^{i\xi (\textbf{w} \cdot \textbf{u})}d \xi d\textbf{u} \int_{\mathbb{R}^{n}} \mathcal{F}_{\alpha} \left( f(\cdot) e^{-i(\cdot)\cdot (\xi \textbf{u})} \right) \left( \textbf{a} \right) K_{-\alpha}(\textbf{w},\textbf{a})  d\textbf{a} \\
&= f(\textbf{w}) \int_{\mathbb{S}^{n-1}} \int_{\mathbb{R}} \overline{\mathcal{F}{\psi}\left( \xi \right)} \mathcal{F}{\eta}\left( \xi \right)  d \xi d\textbf{u} =\left( \eta, \psi \right)  f(\textbf{w}) 
\end{align*}
for a.e. $ \textbf{w} \in \mathbb{R}^n$.
\end{proof}

According to our choice of the standard measure on $\mathbb{Y}^{2n}$, we denote by $L^2(\mathbb{Y}^{2n}):=L^2(\mathbb{Y}^{2n},dbd\textbf{a}d\textbf{u} )$. So, the inner product on this space is
$$\left(F,G \right)_{L^2(\mathbb{Y}^{2n})}:=\int_{\mathbb{S}^{n-1}}\int_{\mathbb{R}^{n}} \int_{\mathbb{R}} F(\textbf{u},b,\textbf{a}) \overline{G(\textbf{u},b,\textbf{a})} dbd\textbf{a}d\textbf{u}.$$
Now, under the conditions of Proposition \ref{prop3.4}, we have
$$\left( f,g \right)= \frac{1}{\left( \eta, \psi \right)}\left( DS_\psi ^{\alpha} f,DS_{\eta} ^{\alpha} g \right)_{L^2(\mathbb{Y}^{2n})}.$$

If $\psi\in \mathcal{S}(\mathbb{R})$ is {a non-trivial window} and $f\in L^1(\mathbb{R}^n)\cap L^2(\mathbb{R}^n) $ then 
$$\| DS_\psi ^{\alpha} f\|_{L^2(\mathbb{Y}^{2n})}=\|\psi\|_{L^2(\mathbb{R}^{n})} \| f\|_{L^2(\mathbb{R}^{n})}.$$
Since $\mathcal{S}(\mathbb{R}^n)$ is a dense subset of $L^2(\mathbb{R}^n)$ and the last relation is true for all $f \in \mathcal{S}(\mathbb{R}^n)$, $DS_{\psi} ^{\alpha}$ can be extended to a constant multiple of an isometric embedding $L^2(\mathbb{R}^n)\to L^2(\mathbb{Y}^{2n})$.

The reconstruction formula suggests us to define the directional short-time fractional Fourier synthesis operator that maps functions on $\mathbb{Y}^{2n}$ to functions on $\mathbb{R}^n$. 
For the given $\psi\in \mathcal{S}(\mathbb{R})$, we define the directional short-time fractional Fourier synthesis operator as
\begin{equation}\label{synthop}
(DS_{\psi} ^{\alpha})^{*} \Phi(\textbf{x}):=\int_{\mathbb{S}^{n-1}}\int_{\mathbb{R}^{n}} \int_{\mathbb{R}} \Phi(\textbf{u},b,\textbf{a}) \psi(\textbf{x} \cdot \textbf{u}-b) K_{-\alpha}(\textbf{x},\textbf{a}) dbd\textbf{a}d\textbf{u}, \enspace \textbf{x}\in \mathbb{R}^n.
\end{equation}
The last integral is absolutely convergent if $\Phi\in \mathcal{S}(\mathbb{Y}^{2n})$.
Now, by using the reconstruction formula, we obtain
$$(DS_{\eta} ^{\alpha})^{*}(DS_\psi ^{\alpha} f)(\textbf{x}):=\int_{\mathbb{S}^{n-1}}\int_{\mathbb{R}^{n}} \int_{\mathbb{R}} DS_\psi ^{\alpha} f(\textbf{u},b,\textbf{a}) \eta(\textbf{x}\cdot \textbf{u}-b) K_{-\alpha}(\textbf{x},\textbf{a}) dbd\textbf{a}d\textbf{u}=\left( \eta, \psi \right)f(\textbf{x}),$$
for a.e. $\textbf{x}\in \mathbb{R}^n$.

We also show that the directional short-time fractional Fourier synthesis operator is in fact the transpose of the DSTFRFT in the following sense:

\begin{proposition}\label{prop3.6}
Let $\psi\in \mathcal{S}(\mathbb{R})$. If $f\in L^1(\mathbb{R}^n)$ and $\Phi\in \mathcal{S}(\mathbb{Y}^{2n})$, then
\begin{equation*}\int_{\mathbb{R}^n}f(\normalfont\textbf{x})\overline{(DS_{\psi} ^{\alpha})^{*}(\overline{\Phi})}(\textbf{x})d\textbf{x}=\int_{\mathbb{S}^{n-1}}\int_{\mathbb{R}^{n}} \int_{\mathbb{R}} DS_\psi ^{\alpha} f(\textbf{u},b,\textbf{a}) \Phi(\textbf{u},b,\textbf{a}) dbd\textbf{a}d\textbf{u}.\end{equation*}
\end{proposition}
\begin{proof}
Using the Fubini's theorem, we obtain
\begin{align*}
\int_{\mathbb{R}^n}f(\textbf{x})\overline{(DS_{\psi} ^{\alpha})^{*}(\overline{\Phi})}(\textbf{x})d\textbf{x} &=\int_{\mathbb{R}^n}f(\textbf{x})d\textbf{x} \int_{\mathbb{S}^{n-1}}\int_{\mathbb{R}^{n}} \int_{\mathbb{R}} \Phi(\textbf{u},b,\textbf{a}) \overline{\psi}(\textbf{x}\cdot \textbf{u}-b)  K_{\alpha}(\textbf{x},\textbf{a}) dbd\textbf{a}d\textbf{u}\\
&=\int_{\mathbb{S}^{n-1}}\int_{\mathbb{R}^{n}} \int_{\mathbb{R}} \Phi(\textbf{u},b,\textbf{a}) dbd\textbf{a}d\textbf{u} \int_{\mathbb{R}^n}f(\textbf{x})\overline{\psi}(\textbf{x}\cdot \textbf{u}-b) K_{\alpha}(\textbf{x},\textbf{a}) d\textbf{x}\\
&=\int_{\mathbb{S}^{n-1}}\int_{\mathbb{R}^{n}} \int_{\mathbb{R}} DS_\psi ^{\alpha} f(\textbf{u},b,\textbf{a}) \Phi(\textbf{u},b,\textbf{a}) dbd\textbf{a}d\textbf{u}.
\end{align*}
\end{proof}
Under the standard identification (\ref{standardidentification}), the last relation takes the form
$$\big\langle f, \overline{(DS_{\psi} ^{\alpha})^{*}(\overline{\Phi})}\, \big\rangle = \langle  DS_{\psi} ^{\alpha }f, \Phi \rangle,$$
that will be our model for defining the distributional DSTFRFT in Section \ref{se4}.


\section{Continuity of the DSTFRFT on test function spaces}\label{se3}

In the next two theorems we show the continuity of $DS_\psi^{\alpha}:\mathcal{S}(\mathbb{R}^n)  \to \mathcal{S}(\mathbb{Y}^{2n})$ and $(DS_{\psi} ^{\alpha})^{*}: \mathcal{S}(\mathbb{Y}^{2n}) \to \mathcal{S}(\mathbb{R}^n)$ provided that $\psi \in \mathcal{S}(\mathbb{R})$.

Notice that we can extend the definition of the DSTFRFT as a sesquilinear mapping $DS^{\alpha}:(f,\psi)\to DS_{\psi} ^{\alpha}f $, whereas the directional short-time fractional Fourier synthesis operator extends to the bilinear form 
$(DS^{\alpha})^{*}:(\Phi,\psi)\to (DS_{\psi}^{\alpha})^{*}\Phi $.

\begin{theorem}\label{ctdstfrft}
The bilinear mapping 
$DS^{\alpha}:\mathcal{S}(\mathbb{R}^n) \times \mathcal{S}(\mathbb{R}) \to \mathcal{S}(\mathbb{Y}^{2n})$ defined as
$\left( f,\psi \right) \to DS_\psi ^{\alpha} f$
is continuous.
\end{theorem}

\begin{proof}
To prove the theorem, it is enough to show that for given $s,r,m,k\in \mathbb{N}_0$ and $l\in \mathbb{N}_0 ^{n}$ there exist $v,\tau \in \mathbb{N}_0$ and $C>0$ such that
$$\rho_{s,r}^{l,m,k}(DS_\psi ^{\alpha} f) \leq C\rho_v(f) \rho_{\tau}(\psi),  \enspace f\in \mathcal{S}(\mathbb{R}^n),\ \psi\in \mathcal{S}(\mathbb{R}).$$
We may suppose that $r$ is even and $s\geq 1.$\\
\textbf{Step 1.} By the definition of the DSTFRFT and the Leibniz’s formula, we obtain
\begin{align*}
&\,\quad \left|\frac{\partial^l}{\partial \textbf{a}^l}\frac{\partial^m}{\partial b^m} DS_{\psi} ^{\alpha} f(\textbf{u},b,\textbf{a})\right|\\
&= \left| \frac{\partial^l}{\partial \textbf{a}^l}\frac{\partial^m}{\partial b^m} \int_{\mathbb{R}^n} {f(\textbf{x}) \overline{\psi}(\textbf{x} \cdot \textbf{u}-b) K_{\alpha}(\textbf{x},\textbf{a})} d\textbf{x} \right| \\
    &= \left| \frac{\partial^l}{\partial \textbf{a}^l} \int_{\mathbb{R}^n} {f(\textbf{x}) \overline{\psi^{(m)}}(\textbf{x} \cdot \textbf{u}-b) K_{\alpha}(\textbf{x},\textbf{a})} d\textbf{x} \right| \\
    &= \left| \sum_{|r|,|d|\leq |l|} {c_{\alpha,r,d}\ \textbf{a}^r \int_{\mathbb{R}^n} {\textbf{x}^d f(\textbf{x}) \overline{\psi^{(m)}}(\textbf{x} \cdot \textbf{u}-b) K_{\alpha}(\textbf{x},\textbf{a})} d\textbf{x}}   \right| 
    \end{align*}
    \begin{align*}
    &\leq \left( 1+|\textbf{a}|^2 \right)^{|l|/2} \sum_{|r|,|d|\leq |l|} {\left| c_{\alpha,r,d} \right| \ \left| \int_{\mathbb{R}^n} {\textbf{x}^d f(\textbf{x}) \overline{\psi^{(m)}}(\textbf{x} \cdot \textbf{u}-b) K_{\alpha}(\textbf{x},\textbf{a})} d\textbf{x} \right|}\\
    &= \left( 1+|\textbf{a}|^2 \right)^{|l|/2} \sum_{|r|,|d|\leq |l|} {\left| c_{\alpha,r,d} \right| \ \left| DS_{\psi_m}^{\alpha} f_d (\textbf{u},b,\textbf{a}) \right|}
\end{align*}
where $\psi_{m}(t)=\psi^{(m)}(t) \in \mathcal{S} (\mathbb{R}), \, f_{d}(\textbf{x})= \textbf{x}^{d}f(\textbf{x}) \in \mathcal{S} (\mathbb{R}^n) $ and $c_{\alpha,r,d}$ is a constant which depends on $\alpha, r,d$. So, we can suppose that $m=0$ and $l=(0,0,...,0)$.\\
\textbf{Step 2.} First, we have
\begin{align*}
    \left| \Delta_{\textbf{u}}^{k} DS_\psi ^{\alpha} f(\textbf{u},b,\textbf{a}) \right| &= \left| 
\Delta_{\textbf{u}}^{k} \left( \int_{\mathbb{R}^n} f(\textbf{x}) \overline{\psi}(\textbf{x} \cdot \textbf{u}-b) K_{\alpha}(\textbf{x},\textbf{a}) d\textbf{x}  \right) \right| \\
&= \left| \sum_{j,|d|\leq 2k} P_{j,d}(\textbf{u}) \int_{\mathbb{R}^n} \textbf{x}^d f(\textbf{x}) \overline{\psi^{(j)}}(\textbf{x} \cdot \textbf{u}-b) K_{\alpha}(\textbf{x},\textbf{a}) d\textbf{x}  \right|\\
& \lesssim \sum_{j,|d|\leq 2k} \left| \int_{\mathbb{R}^n} \textbf{x}^d f(\textbf{x}) \overline{\psi^{(j)}}(\textbf{x} \cdot \textbf{u}-b) K_{\alpha}(\textbf{x},\textbf{a}) d\textbf{x} \right|\\
& \lesssim \sum_{j,|d|\leq 2k} \left| DS_{\psi_{j}} ^{\alpha} f_d (\textbf{u},b,\textbf{a}) \right| ,  
\end{align*}
where the $P_{j,d}(\textbf{u})$ are certain polynomials, $\psi_j(t)=\psi^{(j)}(t)$ and $ f_{d}(\textbf{x})= \textbf{x}^{d}f(\textbf{x}) $. Since  $\psi_{j} \in \mathcal{S} (\mathbb{R})$ and $ f_{d} \in \mathcal{S} (\mathbb{R}^n) $, we can suppose that $k=0$.\\
\textbf{Step 3.} Using Proposition \ref{prop3.2}, we have 
\begin{align*}
    &\,\quad\left( 1+|b|^2 \right)^{r/2} \left| DS_\psi ^{\alpha} f(\textbf{u},b,\textbf{a})\right|\\
    &=\left( 1+|b|^2 \right)^{r/2} \left| (2\pi)^{\frac{n-1}{2}} \int_{\mathbb{R}}{\mathcal{F} \left( f(\cdot)  K_{\alpha}(\cdot,\textbf{a})\right) \left( \xi \textbf{u} \right) \overline{\mathcal{F}{\psi}\left( \xi \right)}  e^{i \xi b} d \xi} \right| \\ 
    &= \bigg|  (2\pi)^{\frac{n-1}{2}} \int_{\mathbb{R}}{\mathcal{F} \left( f(\cdot)  K_{\alpha}(\cdot,\textbf{a})\right) \left( \xi \textbf{u} \right) \overline{\mathcal{F}{\psi}\left( \xi \right)} \left( 1- \frac{\partial ^2}{\partial \xi^2}\right)^{r/2}(e^{i \xi b}) d \xi \bigg|} \\ 
    &=\bigg|(2\pi)^{\frac{n-1}{2}} \int_{\mathbb{R}} \left( 1- \frac{\partial ^2}{\partial \xi^2}\right)^{r/2} \left[ \mathcal{F} \left( f(\cdot)  K_{\alpha}(\cdot,\textbf{a})\right) \left( \xi \textbf{u} \right) \overline{\mathcal{F}{\psi}\left( \xi \right)} \right] e^{i \xi b} d \xi \bigg| \\ 
    &= \bigg| (2\pi)^{\frac{n-1}{2}} \sum_{|d|,j \leq r} Q_{d,j}(\textbf{u}) \int_{\mathbb{R}} \mathcal{F} \left( \textbf{x}^d f(\textbf{x})  K_{\alpha}(\textbf{x},\textbf{a})\right) \left( \xi \textbf{u} \right)  \overline{\mathcal{F}{(t^j \psi(t))}\left( \xi \right)}  e^{i \xi b} d \xi \bigg|, 
\end{align*}
for some polynomials $Q_{d,j}(u)$.
If we put $f_{d}(\textbf{x})=\textbf{x}^{d}f(\textbf{x}), \psi_j(t)=t^j \psi(t)$, using Proposition \ref{prop3.2}, we obtain
\begin{align*}
     &\, \quad\left( 1+|b|^2 \right)^{r/2} \left| DS_\psi ^{\alpha} f(\textbf{u},b,\textbf{a})\right|\\
   &\lesssim \sum_{|d|,j \leq r} \bigg|  (2\pi)^{\frac{n-1}{2}}  \int_{\mathbb{R}} \mathcal{F} \left( f_d(\textbf x)  K_{\alpha}(\textbf{x},\textbf{a}) \right) \left( \xi \textbf{u} \right)  \overline{\mathcal{F}{\psi_{j}}\left( \xi \right)}  e^{i \xi b} d \xi \bigg|\\
   &=\sum_{|d|,j \leq r} \left| DS_{\psi_j} ^{\alpha} f_d(\textbf{u},b,\textbf{a})\right|.   
\end{align*}
Since $\psi_j \in \mathcal{S} (\mathbb{R})$ and $ f_{d} \in \mathcal{S} (\mathbb{R}^n)$, we can suppose that $r=0$.\\
\textbf{Step 4.}
Since $f\in \mathcal{S} (\mathbb{R}^n) $, for $k_1,...,k_n \in \mathbb{N}_0$ and $\textbf{a}=(a_1,...,a_n)\in \mathbb{R}^n$, we obtain 
\begin{align*}
    &\, \quad \prod_{r=1}^{n} {|a_r|^{k_r} \left| DS_{\psi} ^{\alpha}f(\textbf{u},b,\textbf{a}) \right|}= \prod_{r=1}^{n} {|a_r|^{k_r} \left| \int_{\mathbb{R}^n}{f(\textbf{x})\overline{\psi}(\textbf{u}\cdot \textbf{x}-b) K_{\alpha}(\textbf{x},\textbf{a})d\textbf{x}} \right|}\\
    &\lesssim \prod_{r=1}^{n} {|a_r|^{k_r} \left| \int_{\mathbb{R}^n}{f(\textbf{x}) e^{i \frac{c_1 \cdot (x_1 ^2 ,...,x_n ^2)}{2}} \overline{\psi}(\textbf{u}\cdot \textbf{x}-b)  e^{-i  \textbf{x} \cdot (a_1 c_2 ^{(1)},...,a_n c_2 ^{(n)})}d\textbf{x}} \right|}\\
    &\lesssim \left| \int_{\mathbb{R}^n}{f(\textbf{x}) e^{i \frac{c_1 \cdot (x_1 ^2 ,...,x_n ^2)}{2}} \overline{\psi}(\textbf{u}\cdot \textbf{x}-b) \frac{\partial^{k_1+...+k_n}}{\partial x_1 ^{k_1} \cdots \partial x_n ^{k_n} } \left( e^{-i  \textbf{x} \cdot (a_1 c_2 ^{(1)},...,a_n c_2 ^{(n)})} \right) d\textbf{x}} \right|\\
    &= \left| \int_{\mathbb{R}^n}{ \frac{\partial^{k_1+...+k_n}}{\partial x_1 ^{k_1} \cdots \partial x_n ^{k_n} } \left( f(\textbf{x}) e^{i \frac{c_1 \cdot (x_1 ^2 ,...,x_n ^2)}{2}} \overline{\psi}(\textbf{u}\cdot \textbf{x}-b) \right)   e^{-i  \textbf{x} \cdot (a_1 c_2 ^{(1)},...,a_n c_2 ^{(n)})}  d\textbf{x}} \right|\\
 & \lesssim \rho_{q_1}(f) \rho_{q_2}(\psi)  
\end{align*}
for some $q_1,q_2 \in \mathbb{N}_0$. 
\end{proof}

We now analyze the directional short-time fractional Fourier synthesis operator.

\begin{theorem} \label{ctdstfrft1}
The bilinear mapping 
$(DS^{\alpha})^{*}: \mathcal{S}(\mathbb{Y}^{2n}) \times \mathcal{S}(\mathbb{R}) \to \mathcal{S}(\mathbb{R}^n)$ defined as 
$(\Phi,\psi) \to (DS_{\psi} ^{\alpha})^{*}\Phi$
is continuous.
\end{theorem}

\begin{proof}
Since $\Phi \in \mathcal{S}(\mathbb{Y}^{2n}) $  and $\psi \in \mathcal{S}(\mathbb{R}) $, it follows that  $(DS_{\psi} ^{\alpha})^{*}\Phi \in \mathbb{C}^{\infty}(\mathbb{R}^n).$
Using \eqref{synthop} and the proof of Theorem 3.2 in \cite{APS}, we  have 
\begin{align*}
(DS_{\psi} ^{\alpha})^{*}\Phi(\textbf{x})&=\int_{\mathbb{S}^{n-1}}\int_{\mathbb{R}^{n}} \int_{\mathbb{R}} \Phi(\textbf{u},b,\textbf{a}) \psi(\textbf{x} \cdot \textbf{u}-b) K_{-\alpha}(\textbf{x},\textbf{a}) dbd\textbf{a}d\textbf{u}\\
& = \int_{\mathbb{S}^{n-1}} d\textbf{u} \int_{\mathbb{R}^{n}} K_{-\alpha}(\textbf{x},\textbf{a}) d\textbf{a} \int_{\mathbb{R}} \hat{\Phi}(\textbf{u},w,\textbf{a}) \hat{\psi}(w) e^{iw(\textbf{x}\cdot \textbf{u})} dw, 
\end{align*}
where $\hat{\Phi}$ stands for the FT of $\Phi(u,b,a)$ with respect to the variable $b$. Now, using the Fubini's theorem, we obtain
\begin{align*}
(DS_{\psi} ^{\alpha})^{*}\Phi(\textbf{x})
& =\int_{\mathbb{R}} \hat{\psi}(w)  dw \int_{\mathbb{S}^{n-1}} e^{iw(\textbf{x}\cdot \textbf{u})} d\textbf{u} \int_{\mathbb{R}^{n}} \hat{\Phi}(\textbf{u},w,\textbf{a}) K_{-\alpha}(\textbf{x},\textbf{a}) d\textbf{a}  \\
& =\int_{\mathbb{R}} \hat{\psi}(w)  dw \int_{\mathbb{S}^{n-1}} e^{iw(\textbf{x}\cdot \textbf{u})}  \mathcal{F}_{-\alpha}{\left( \hat{\Phi}(\textbf{u},w,\cdot) \right)(\textbf{x})} d\textbf{u} . 
\end{align*}
Since $\mathcal{F}_{-\alpha}: \mathcal{S}(\mathbb{R}^n)\to \mathcal{S}(\mathbb{R}^n)$ is an isomorphism, we conclude that for all $v\in \mathbb{N}_0$
$$\rho_{v}\left((DS_{\psi} ^{\alpha})^{*}\Phi \right)=\sup_{\textbf{x}\in \mathbb{R}^n,   |\beta| \leq v } (1+|\textbf{x}|)^{v} \left| \frac{\partial ^{|\beta|}}{\partial \textbf{x}^{\beta}} (DS_{\psi} ^{\alpha})^{*}\Phi (\textbf{x}) \right| \lesssim \rho_{s,r}^{l,m,k}(\Phi) \rho_{q}(\psi),$$
for some $s,r,m,k,q \in \mathbb{N}_0$ and $l\in \mathbb{N}_0 ^n.$
\end{proof}

The next proposition is a direct consequence of Proposition \ref{prop3.5}, Theorem \ref{ctdstfrft}, and Theorem \ref{ctdstfrft1}.

\begin{proposition} \label{prop4.3}
Let $\psi\in \mathcal{S} (\mathbb{R})$ be a non-trivial window and $\eta\in \mathcal{S}(\mathbb{R})$ be a synthesis window for it. Then
\begin{equation} \label{RFormulaDSTFrFT}
\frac{1}{\left( \eta, \psi \right)}\Big((DS_{\eta} ^{\alpha})^{*} \circ DS_\psi ^{\alpha}\Big)= Id_{\mathcal{S} (\mathbb{R}^n)}.
\end{equation} 

\end{proposition}


\section{The DSTFRFT on $\mathcal{S}'(\mathbb{R}^n)$}\label{se4}

We are ready to define the DSTFRFT and the directional fractional short-time fractional Fourier synthesis operator of tempered distributions.

\begin{definition} \label{def5.1}
Let $\psi\in \mathcal{S} (\mathbb{R})$. We define the DSTFRFT of $f\in \mathcal{S}' (\mathbb{R}^n)$ with respect to $\psi$ as the element $DS_\psi ^{\alpha} f \in \mathcal{S}'(\mathbb{Y}^{2n})$ whose action on test functions is given by 
$$\langle DS_{\psi}^{\alpha} f,\Phi \rangle:= \big\langle\, f, \overline{(DS_\psi ^{\alpha}) ^{*}(\overline{\Phi})}\,\big\rangle, \enspace \Phi\in \mathcal{S}(\mathbb{Y}^{2n}).$$\end{definition}
This definition is valid according to Theorem \ref{ctdstfrft1}.

\begin{remark} \label{rem3}
Using Proposition \ref{prop3.6}, one can show that Definition \ref{def3.1} and Definition \ref{def5.1} coincide for $f\in L^1(\mathbb{R}^n)$. It means that
$$\langle DS_{\psi}^{\alpha} f,\Phi \rangle =\int _{\mathbb{S}^{n-1}} \int_{\mathbb{R}^{n}} \int_{\mathbb{R}} DS_{\psi} ^{\alpha} (\textbf{u},b,\textbf{a}) \Phi(\textbf{u},b,\textbf{a}) dbd\textbf{a}d\textbf{u}, \enspace \Phi\in \mathcal{S}(\mathbb{Y}^{2n}).$$
\end{remark}

\begin{definition} \label{def5.2}
Let $\psi\in \mathcal{S} (\mathbb{R})$. We define the directional fractional short-time fractional Fourier synthesis operator 
$(DS_\psi ^{\alpha}) ^{*}:\mathcal{S}'(\mathbb{Y}^{2n}) \to \mathcal{S}' (\mathbb{R}^n)  $ as
$$\langle (DS_\psi ^{\alpha}) ^{*} F,\varphi \rangle:=\big\langle F,\overline{DS_\psi ^{\alpha} (\overline{\varphi})}\,\big \rangle, \enspace F \in \mathcal{S}'(\mathbb{Y}^{2n}), \enspace \varphi \in \mathcal{S} (\mathbb{R}^n). $$\end{definition}
The validity of Definition \ref{def5.2} is guaranteed by Theorem \ref{ctdstfrft}.

The following continuity result is obtained by taking transposes in Theorems \ref{ctdstfrft} and Theorem \ref{ctdstfrft1}. 

\begin{proposition}
Let $\psi \in \mathcal{S} (\mathbb{R})$. The DSTFRFT $DS_{\psi}^{\alpha}:\mathcal{S}' (\mathbb{R}^n) \to \mathcal{S}'(\mathbb{Y}^{2n}) $ and the directional short-time fractional Fourier synthesis operator $(DS_\psi ^{\alpha}) ^{*}:\mathcal{S}'(\mathbb{Y}^{2n}) \to \mathcal{S}' (\mathbb{R}^n)  $ are continuous linear maps.
\end{proposition}
We can generalize the reconstruction formula (\ref{RFormulaDSTFrFT}) for the space of tempered distributions.
\begin{proposition} \label{prop5.4}
Let $\psi\in \mathcal{S} (\mathbb{R})$ be a non-trivial window and $\eta \in \mathcal{S} (\mathbb{R})$ be a synthesis window for it. Then 
$$\frac{1}{\left( \eta, \psi \right)}\Big((DS_\eta ^{\alpha}) ^{*} \circ DS_\psi ^{\alpha}\Big)= Id_{\mathcal{S}'(\mathbb{R}^n)}.$$
\end{proposition}
\begin{proof}
Let $f\in \mathcal{S}'(\mathbb{R}^n)$ and $\varphi \in \mathcal{S}(\mathbb{R}^n)$. Using Proposition \ref{prop4.3}, we obtain 
\begin{align*}
    \langle (DS_\eta ^{\alpha}) ^{*} \circ DS_\psi ^{\alpha}\ f,\varphi \rangle&=\big\langle DS_\psi ^{\alpha} f,\overline{DS_{\eta} ^{\alpha} (\overline{\varphi})} \,\big\rangle=\big\langle f,\overline{(DS_\psi ^{\alpha}) ^{*} \circ (DS_{\eta} ^{\alpha}(\overline{\varphi}))} \,\big\rangle\\&=\langle f, \overline{\left( \psi, \eta \right)} \varphi\rangle=\left( \eta, \psi \right)\langle f,\varphi \rangle.
\end{align*}
\end{proof}


\section{Desingularization formula}\label{se5}

Let $f\in \mathcal{S} (\mathbb{R}^{n})$ and $\psi\in \mathcal{S}(\mathbb{R}) $. By Remark \ref{rem3}, for $\Phi \in \mathcal{S}(\mathbb{Y}^{2n})$, we obtain
\begin{align*}
\langle DS_{\psi}^{\alpha} f,\Phi \rangle &=\int _{\mathbb{S}^{n-1}} \int_{\mathbb{R}^{n}} \int_{\mathbb{R}} DS_{\psi} ^{\alpha} f(\textbf{u},b,\textbf{a}) \Phi(\textbf{u},b,\textbf{a}) dbd\textbf{a}d\textbf{u}\\
&=\int _{\mathbb{S}^{n-1}} \int_{\mathbb{R}^{n}} \int_{\mathbb{R}} \mathcal{F}_{\alpha} \left( f(\cdot)\overline{\psi}((\cdot)\cdot \textbf{u}-b) \right)(\textbf{a}) \Phi(\textbf{u},b,\textbf{a}) dbd\textbf{a}d\textbf{u}\\
&=\int_{\mathbb{\mathbb{S}}^{n-1}} \int_{\mathbb{R}}  \langle \mathcal{F}_{\alpha} \left( f(\cdot)\overline{\psi}((\cdot)\cdot \textbf{u}-b) \right)(\textbf{a}), \Phi(\textbf{u},b,\textbf{a}) \rangle_\textbf{a} dbd\textbf{u}.
\end{align*}

In the case when $f\in \mathcal{S}' (\mathbb{R}^{n}) $ 
 and $\psi\in \mathcal{S}(\mathbb{R})$, one can show that $f(\cdot)\overline{\psi}((\cdot)\cdot \textbf{u}-b) \in \mathcal{S}' (\mathbb{R}^{n})$, for $(\textbf{u},b) \in \mathbb{S}^{n-1} \times \mathbb{R}$.  Consequently, $\mathcal{F}_{\alpha} \left( f(\cdot)\overline{\psi}((\cdot)\cdot \textbf{u}-b) \right)\in \mathcal{S}' (\mathbb{R}^{n})$, which motivated us for the next result that gives a relation between the DSTFRFT and the FRFT.

\begin{proposition} \label{prop6.1}
Let $f\in \mathcal{S} ' (\mathbb{R}^{n}) $ and $\psi\in \mathcal{S}(\mathbb{R}) $. Then
\begin{equation} \label{8}
\langle DS_\psi ^{\alpha} f, \Phi \rangle=\int_{\mathbb{\mathbb{S}}^{n-1}} \int_{\mathbb{R}}  \langle \mathcal{F}_{\alpha} \left( f(\cdot)\overline{\psi}((\cdot)\cdot \textbf{u}-b) \right)(\textbf{a}), \Phi(\textbf{u},b,\textbf{a}) \rangle_\textbf{a} dbd\textbf{u}, \enspace \Phi \in \mathcal{S}(\mathbb{Y}^{2n}).
\end{equation}
Furthermore, $DS_\psi ^{\alpha} f \in \mathbb{C}^{\infty}(\mathbb{S}^{n-1} \times \mathbb{R}, \mathcal{S}'(\mathbb{R}^{n}))$ and it is of slow growth on $\mathbb{S}^{n-1} \times \mathbb{R}$.
\end{proposition}
 \begin{proof}
Using Definition \ref{def5.1} and the Fubini's theorem, we obtain
\begin{align*}
\langle DS_{\psi}^{\alpha} f,\Phi \rangle &= \big\langle f(\textbf x), \overline{(DS_{\psi} ^{\alpha})^{*}(\overline{\Phi})}(\textbf x)\,\big\rangle \\
&=\big\langle f(\textbf x), \int_{\mathbb{S}^{n-1}}\int_{\mathbb{R}^{n}} \int_{\mathbb{R}} \Phi(\textbf{u},b,\textbf{a}) \overline{\psi}(\textbf{x} \cdot \textbf{u}-b) K_{\alpha}(\textbf{x},\textbf{a}) dbd\textbf{a}d\textbf{u}\big\rangle
 \end{align*}
\begin{align*}
&=\langle f(\textbf{x}), \int_{\mathbb{S}^{n-1}}\int_{\mathbb{R}}  \overline{\psi}(\textbf{x} \cdot \textbf{u}-b) \mathcal{F}_{\alpha} \left( \Phi (\textbf{u},b,\cdot)\right)(\textbf{x}) dbd\textbf{u}\rangle.
\end{align*}

Since $f\in \mathcal{S} ' (\mathbb{R}^{n})$, it follows that $f=\partial ^{\beta}h$, for some continuous function $h$ of at most polynomial growth on $\mathbb{R}^{n}$, and some $\beta \in \mathbb{N}_{0}^{n}$ (\cite{Sch66}, Thm. VI, page 239).
Then, using the Fubini's theorem, we obtain 
\begin{align*}
 \quad \langle DS_{\psi}^{\alpha} f,\Phi \rangle&=\langle f(\textbf{x}), \int_{\mathbb{S}^{n-1}}\int_{\mathbb{R}}  \overline{\psi}(\textbf{x} \cdot \textbf{u}-b) \mathcal{F}_{\alpha} \left( \Phi (\textbf{u},b,\cdot)\right)(\textbf{x}) dbd\textbf{u}\rangle\\
&=(-1)^{|\beta|} \int_{\mathbb{R}^n}{h(\textbf{x}) \frac{\partial^{|\beta|}}{\partial \textbf{x}^{\beta}} \left( \int_{\mathbb{S}^{n-1}}\int_{\mathbb{R}}  \overline{\psi}(\textbf{x} \cdot \textbf{u}-b) \mathcal{F}_{\alpha} \left( \Phi (\textbf{u},b,\cdot)\right)(\textbf{x}) dbd\textbf{u} \right)}d\textbf{x}\\
&=(-1)^{|\beta|} \int_{\mathbb{S}^{n-1}}\int_{\mathbb{R}} dbd\textbf{u} \int_{\mathbb{R}^n}{h(\textbf{x}) \frac{\partial^{|\beta|}}{\partial \textbf{x}^{\beta}} \left( \overline{\psi}(\textbf{x} \cdot \textbf{u}-b) \mathcal{F}_{\alpha} \left( \Phi (\textbf{u},b,\cdot) \right)(\textbf{x}) \right)} d\textbf{x}\\
&=\int_{\mathbb{S}^{n-1}}\int_{\mathbb{R}} \langle f(\textbf{x}),\overline{\psi}(\textbf{x} \cdot \textbf{u}-b) \mathcal{F}_{\alpha} \left( \Phi (\textbf{u},b,\cdot) \right)(\textbf{x}) \rangle _{\textbf{x}} dbd\textbf{u}\\
&=\int_{\mathbb{\mathbb{S}}^{n-1}} \int_{\mathbb{R}}  \langle \mathcal{F}_{\alpha} \left( f(\cdot)\overline{\psi}((\cdot)\cdot \textbf{u}-b) \right)(\textbf{a}), \Phi(\textbf{u},b,\textbf{a}) \rangle_\textbf{a} dbd\textbf{u},
\end{align*}
which proves the relation (\ref{8}). One can show that for fixed $\varphi \in \mathcal{S}(\mathbb{R}^n)$, the function
$$(\textbf{u},b)\to \langle \mathcal{F}_{\alpha} \left( f(\cdot)\overline{\psi}((\cdot)\cdot \textbf{u}-b) \right)(\textbf{a}), \varphi(\textbf{a}) \rangle_\textbf{a}, \quad (\textbf{u},b) \in \mathbb{S}^{n-1} \times \mathbb{R},$$
is smooth of at most polynomial growth on $\mathbb{S}^{n-1} \times \mathbb{R}$.

For $\varphi \in \mathbb{S}(\mathbb{R}^n)$ and $\Psi \in \mathbb{S}(\mathbb{S}^{n-1} \times \mathbb{R})$, under the standard identification (\ref{SI}), we obtain
\begin{align*}
&\, \quad \langle \langle DS_{\psi}^{\alpha}f(\textbf{u},b,\textbf{a}), \varphi(\textbf{a}) \rangle_{\textbf{a}}, \Psi(\textbf{u},b)\rangle_{\textbf{u},b}= \langle DS_{\psi}^{\alpha}f, \varphi \Psi \rangle\\
&=\int_{\mathbb{\mathbb{S}}^{n-1}} \int_{\mathbb{R}}  \langle \mathcal{F}_{\alpha} \left( f(\cdot)\overline{\psi}((\cdot)\cdot \textbf{u}-b) \right)(\textbf{a}), \varphi(\textbf{a}) \rangle_\textbf{a} \Psi(\textbf{u},b) dbd\textbf{u}.
\end{align*}
Using the standard identification, we obtain
\begin{align*}
    \langle DS_{\psi}^{\alpha}f(\textbf{u},b,\textbf{a}), \varphi(\textbf{a}) \rangle_{\textbf{a}}&=\langle\mathcal{F}_{\alpha} \left( f(\cdot)\overline{\psi}((\cdot)\cdot \textbf{u}-b) \right)(\textbf{a}), \varphi(\textbf{a}) \rangle_\textbf{a}\\
    &=\langle f(\textbf{x}),\overline{\psi}(\textbf{x} \cdot \textbf{u}-b) \mathcal{F}_{\alpha} \varphi (\textbf{x}) \rangle_\textbf{x}.
\end{align*}
Then $DS_\psi ^{\alpha} f \in \mathbb{C}^{\infty}(\mathbb{S}^{n-1} \times \mathbb{R}, \mathcal{S}'(\mathbb{R}^{n}))$ and it is of slow growth on $\mathbb{S}^{n-1} \times \mathbb{R}$.
 \end{proof}

As a corollary of Proposition \ref{prop6.1}, we end this article with the following desingularization formula.
 
 \begin{corollary}\textnormal{(Desingularization formula)} \label{desingform}
 Let $f\in \mathcal{S}' (\mathbb{R}^{n}) $ and $\psi\in \mathcal{S} (\mathbb{R}) $ be a non-trivial window. If $\eta \in \mathcal{S} (\mathbb{R}) $ is a synthesis window for $\psi$, then 
\begin{equation}
\langle f, \varphi \rangle= \frac{1}{\left( \eta, \psi \right)} \int_{\mathbb{\mathbb{S}}^{n-1}} \int_{\mathbb{R}}  \langle \mathcal{F}_{\alpha} \left( f(\cdot)\overline{\psi}((\cdot)\cdot \textbf{u}-b) \right)(\textbf{a}), \overline{DS_{\eta}^{\alpha}(\overline{\varphi})}(\textbf{u},b,\textbf{a}) \rangle_\textbf{a} dbd\textbf{u} ,  \enspace \varphi \in \mathcal{S} (\mathbb{R}^{n}).
\end{equation}
 \end{corollary}
 \begin{proof}
 By Definition \ref{def5.2}, Proposition \ref{prop5.4}, and Proposition \ref{prop6.1}, we have
 \begin{align*}
 \langle f, \varphi \rangle&= \frac{1}{\left( \eta, \psi \right)} \langle (DS_{\eta} ^{\alpha})^{*}(DS_\psi ^{\alpha} f), \varphi \rangle = \frac{1}{\left( \eta, \psi \right)} \big\langle DS_\psi ^{\alpha} f, \overline{DS_{\eta} ^{\alpha}(\overline{\varphi })}\,\big\rangle\\
 &=\frac{1}{\left( \eta, \psi \right)} \int_{\mathbb{\mathbb{S}}^{n-1}} \int_{\mathbb{R}} \big \langle \mathcal{F}_{\alpha} \left( f(\cdot)\overline{\psi}((\cdot)\cdot \textbf{u}-b) \right)(\textbf{a}), \overline{DS_{\eta}^{\alpha}(\overline{\varphi})}(\textbf{u},b,\textbf{a}) \,\big\rangle_\textbf{a} dbd\textbf{u}.
 \end{align*}
\end{proof}

\section*{Disclosure statement}

The authors report there are no competing interests to declare.

\end{document}